\documentclass[12pt,a4]{amsart}
\usepackage[a4paper, left=28mm, right=28mm, top=28mm, bottom=34mm]{geometry}
\usepackage{amsmath}
\usepackage{amssymb}
\usepackage{amsthm}
\usepackage{bm}
\newtheorem{dfn}{Definition}[section]
\newtheorem{thm}[dfn]{Theorem}
\newtheorem{prop}[dfn]{Proposition}

\newtheorem{cor}[dfn]{Corollary}
\newtheorem{exm}[dfn]{Example}
\newtheorem{rem}[dfn]{Remark}
\newtheorem{conj}[dfn]{Conjecture}

\newtheorem{claim}[dfn]{Claim}

\usepackage{amscd}
\usepackage{mathrsfs}

\def\Q{\mathbb{Q}}

\def\Z{\mathbb{Z}}

\def\F{\mathbb{F}}

\def\Spec{\mathop{\mathrm{Spec}}\nolimits}
\def\Pic{\mathop{\mathrm{Pic}}\nolimits}
\def\hatBr{\mathop{\widehat{\mathrm{Br}}}\nolimits}

\def\Spf{\mathop{\mathrm{Spf}}\nolimits}
\def\L{\mathop{\mathscr{L}}\nolimits}
\def\M2dh{\mathop{M^{(h)}_{2d}}\nolimits}
\def\M2d{\mathop{M_{2d}}\nolimits}
\def\g1{\mathop{\gamma_1}\nolimits}
\def\g2{\mathop{\gamma_2}\nolimits}
\def\O{\mathop{\mathscr{O}}\nolimits}

\title[Existence of supersingular reduction for families of $K3$ surfaces]{Existence of supersingular reduction for families of $K3$ surfaces with large Picard number in positive characteristic}
\address{Department of Mathematics\\ Faculty of Science\\ Kyoto University\\ Kyoto 606-8502, Japan}
\email{kito@math.kyoto-u.ac.jp}
\date{\today}
\author{Kazuhiro Ito}

\begin{document}
\maketitle
\begin{abstract}
 We study non-isotrivial families of $K3$ surfaces in positive characteristic $p$ whose geometric generic fibers satisfy $\rho\geq21-2h$ and $h\geq3$, where $\rho$ is the Picard number and $h$ is the height of the formal Brauer group. We show that, under a mild assumption on the characteristic of the base field, they have potential supersingular reduction. Our methods rely on Maulik's results on moduli spaces of $K3$ surfaces and the construction of sections of powers of Hodge bundles due to van der Geer and Katsura. For large $p$ and each $2\leq{h}\leq10$, using deformation theory and Taelman's methods, we construct non-isotrivial families of $K3$ surfaces satisfying $\rho=22-2h$. 

 \end{abstract}
 
 \section{Introduction}\label{intro}
  We study the variation of heights in non-isotrivial families of $K3$ surfaces in characteristic $p>0$ whose geometric generic fibers have large Picard number. Recall that a \textit{$K3$ surface} $X$ over a field is a projective smooth surface with trivial canonical bundle and  $H^1(X,\O_X)=0$. Let $k$ be an algebraically closed field of positive characteristic $p>0$. For a $K3$ surface $X$ over $k$, let $h(X)$ be the \textit{height} of the formal Brauer group. (When $k$ is not algebraically closed, the height of $X$ is defined to be the height of $X_{\overline{k}}:=X\otimes_{k}{\overline{k}}$.) We have $1\leq{h(X)}\leq10$ or $h(X)=\infty.$ When $h(X)\neq\infty$, the Artin-Mazur-Igusa inequality $$\rho(X)\leq22-2h(X)$$ is satisfied \cite[Theorem 0.1]{Artin74}, where $\rho(X)$ is the Picard number of $X$. A $K3$ surface $X$ over $k$ is called \textit{supersingular} if $h(X)=\infty.$ The Tate conjecture for $K3$ surfaces \cite{Charles13}, \cite{Kim-Madapusi}, \cite{Madapusi}, \cite{Maulik}, \cite{Nygaard}, \cite{Nygaard2} implies that $X$ is supersingular if and only if $\rho(X)=22$. (See also \cite[Corollaire 0.5]{Benoist}, \cite[Corollary 17.3.7]{Huybrechts}.) %(Recently, Kim and Madapusi Pera proved the case of characteristic $2$ \cite{Kim-Madapusi}.) 
  \par
   Let $C$ be a proper smooth curve over $k$ with function field $K:=k(C)$. Let $X$ be a $K3$ surface over $K$. Let $v\in{C}$ be a closed point where $X$ has potential good reduction. We say the \textit{height of $X$ jumps at $v$} (resp.\ $X$ has \textit{potential supersingular reduction} at $v$) if there exist a finite extension $L/K$ and a valuation $w$ of $L$ extending the valuation of $v$, and a smooth family of $K3$ surfaces $\mathscr{X}$ over the valuation ring $\mathscr{O}_w$ such that $\mathscr{X}\otimes_{\mathscr{O}_w}{L}\simeq{X\otimes_{K}{L}}$, and $h(\mathscr{X}_s)>h(X)$ (resp.\ $h(\mathscr{X}_s)=\infty$). Here, $\mathscr{X}_s$ is the special fiber of $\mathscr{X}$. We say $X$ is \textit{non-isotrivial} if there does not exist a $K3$ surface $Y$ over $k$ such that $X\otimes_{K}{\overline{K}}\simeq{Y\otimes_{k}{\overline{K}}}$.  \par
      
     The first main result of this paper is as follows.
 \begin{thm}\label{main}
 Let $k$ be an algebraically closed field of characteristic $p>0$, $C$ a projective smooth curve with function field $K:=k(C)$. Let $X$ be a non-isotrivial $K3$ surface over $K$ that admits an ample line bundle $\mathscr{L}$ of degree $2d$. Assume that $p>18d+4$ and $3\leq{h(X)}\leq10$. Then the height of $X$ jumps at some closed point $v\in{C}$.
%If $3\leq{h(X)}\leq10$ and $\rho(X_{\overline{K}})\geq21-2h(X)$, then the $K3$ surface $X$ has potential supersingular reduction at some closed point $v\in{C}$.
\end{thm}
 
As a direct consequence of Theorem \ref{main} and the Artin-Mazur-Igusa inequality $\rho\leq 22-2h$, we have the following corollary.
 \begin{cor}\label{supersingular-reduction}
 Under the assumptions of Theorem \ref{main}, assume moreover that $\rho(X_{\overline{K}})\geq21-2h(X)$. Then $X$ has potential supersingular reduction at some closed point $v\in{C}$. \end{cor}
 The following theorem is the second main result of this paper. We shall prove the existence of non-isotrivial $K3$ surfaces over function fields with $\rho=22-2h$ if $h \geq2$.
\begin{thm}\label{main2}
Let $p$ be a prime number and $h$ a positive integer with $2\leq{h}\leq10$. There exist non-isotrivial $K3$ surfaces $X$ over function fields of characteristic $p$ satisfying $\rho(X)=22-2h(X)$ and $h(X)=h$ if at least one of the following conditions holds:
\begin{itemize}
\item $p=3$ and $h=10$, or
\item $p\geq5$.
\end{itemize}
\end{thm}

\begin{rem}
When $p\geq 3$, Jang and Liedtke independently proved that there do not exist non-isotrivial $K3$ surfaces $X$ with $\rho(X)=22-2h(X)$ and $h(X)=1$ \cite[Theorem 3.7]{Jang}, \cite[Theorem 2.6]{Liedtke}. See also Theorem \ref{20}.
\end{rem}

 \begin{rem}
\rm{The main reason why we assume conditions on prime $p$ in Theorem \ref{main} and Theorem \ref{main2} is that we do not currently know the existence of potential semistable reduction of $K3$ surfaces over discretely valued fields. For the precise statement we need, see \cite[Assumption ($\star$)]{Liedtke-Matsumoto}. If we assume \cite[Assumption ($\star$)]{Liedtke-Matsumoto}, we can show Theorem \ref{main} for $p\geq 3$ and Theorem \ref{main2} for any $p$. See also Remark \ref{maincond}.}
\end{rem} 
 
The outline of this paper is as follows. We recall the results of van der Geer-Katsura \cite{van-katsura} and Maulik \cite{Maulik} on the variation of heights in families of $K3$ surfaces in Section \ref{height}. The proof of Theorem \ref{main} is given in Section \ref{proof}. In Section \ref{deformation}, using deformation theory, we show Theorem \ref{main2} comes down to show the existence of $K3$ surfaces $X$ over $\overline{\F}_p$ satisfying $\rho(X)=22-2h(X)$ and $h(X)=h$. Assuming semistable reduction, Taelman conditionally proved the existence of $K3$ surfaces over finite fields with given $L$-function, up to finite extensions of the base field; see \cite{Taelman}. When the characteristic of the base field is not too small, the author recently proved that Taelman's results hold unconditionally; see \cite{Ito}. Using these results, in Section \ref{existence}, we construct $K3$ surfaces $X$ over $\overline{\F}_p$ satisfying $\rho(X)=22-2h(X)$ and $h(X)=h$ for large $p$ and each $2\leq{h}\leq10$; see Proposition \ref{2,5}. Then we achieve Theorem \ref{main2}.

\section{The variation of heights in families of $K3$ surfaces}\label{height}
Let $X$ be a $K3$ surface over an algebraically closed field $k$ of characteristic $p>0$. The following functor from the category $\mathrm{Art}_k$ of local artinian $k$-algebras with residue field $k$ to the category of abelian groups
$$ 
\begin{CD}
 {\Phi^2_X}\colon    &{\mathrm{Art}_k}   @>>>  &{\text{(Abelian  groups)}}\\
                         & {R}      @>>>  &\ker({H^2_{\rm{\acute{e}t}}(X\otimes_k{R},\mathbb{G}_m)}\rightarrow{H^2_{\rm{\acute{e}t}}(X,\mathbb{G}_m)})\\
\end{CD}
$$
is pro-representable by a smooth one-dimensional formal group scheme $\hatBr(X)$ \cite{Artin-Mazur}. The \textit{height} of $X$ is defined to be the height of $\hatBr(X)$
$$h(X):=h({\hatBr(X)}).$$ 
We have $1\leq{h(X)}\leq10$ or $h(X)=\infty$. When ${h(X)}=\infty$, we say $X$ is \textit{supersingular}. When $h(X)\neq\infty$, Artin proved the following inequality
  $$\rho(X)\leq22-2h(X),$$ where $\rho(X)$ is the Picard number of $X$ \cite[Theorem 0.1]{Artin74}. The Tate conjecture for $K3$ surfaces \cite{Charles13}, \cite{Kim-Madapusi}, \cite{Madapusi}, \cite{Maulik}, \cite{Nygaard}, \cite{Nygaard2} implies that $X$ is supersingular if and only if $\rho(X)=22$. (See also \cite[Corollaire 0.5]{Benoist}, \cite[Corollary 17.3.7]{Huybrechts}.) When $k$ is not algebraically closed, the height of $X$ is defined to be the height of $X\otimes_{k}{\overline{k}}$.\par
  We say $f\colon \mathscr{X} \to S$ is a family of $K3$ surfaces if $S$ is a scheme, $\mathscr{X}$ is an algebraic space, and $f$ is a proper smooth morphism whose geometric fibers are $K3$ surfaces. A \textit{polarization} (resp.\ \textit{quasi-polarization}) of $f\colon \mathscr{X} \to S$ is a section ${\xi}\in\underline{\Pic}(\mathscr{X}/S)(S)$ of the relative Picard functor whose fiber $\xi(s)$ at every geometric point $s \to {S}$ is a polarization (resp.\ quasi-polarization), which means an ample (resp.\ big and nef) line bundle on the $K3$ surface $\mathscr{X}_s$. We say a section ${\xi}\in\underline{\Pic}(\mathscr{X}/S)(S)$ is \textit{primitive} if, for every geometric point $s\to{S}$, the cokernel of the inclusion $\langle\xi(s)\rangle\hookrightarrow\Pic(\mathscr{X}_s)$ is torsion free. For an integral scheme $S$ over $k$, we say a family of $K3$ surfaces $f\colon \mathscr{X}\to{S}$ is \textit{non-isotrivial} if there do not exist a $K3$ surface $Y$ over $k$ and a finite flat morphism $S' \to S$ such that $\mathscr{X}\times_{S}{S'}\simeq{Y\times_{\Spec{k}}{S'}}$.\par

  For a positive integer $d\geq1$, let $M_{2d}$ (resp.\ ${M^{\circ}_{2d}}$) be the moduli functor that sends a $\Z$-scheme $S$ to the groupoid consists of tuples $(f\colon \mathscr{X} \to S, \xi)$, where $f\colon \mathscr{X} \to S$  is a family of $K3$ surfaces and $\xi\in\underline{\Pic}(\mathscr{X}/S)(S)$ is a primitive quasi-polarization (resp.\ primitive polarization) of degree $2d$ (i.e.\ for every geometric point $s \to {S}$, $({\xi(s)})^2=({\xi(s)},{\xi(s)})=2d$, where $(\ ,\ )$ denotes the intersection pairing on $\mathscr{X}_s$).\par
     The following theorem is well-known.
  \begin{thm}[{\cite[Theorem 4.3.4]{Rizov06}, \cite[Proposition 2.1]{Maulik}}]
  The moduli functors $M_{2d}$ and ${M^{\circ}_{2d}}$ are Deligne-Mumford stacks of finite type over $\Spec{\Z}$, and smooth over $\Spec{\Z}[1/2d]$.
  \end{thm}
  In the following, we fix an algebraically closed field $k$ of characteristic $p>0$ and a positive integer $d\geq1$ such that $p$ does not divide $2d$. We work with the moduli stacks $M_{2d, k}:={M_{2d}}\otimes_{\Z}{k}$ and $M^{\circ}_{2d, k}:={M^{\circ}_{2d}}\otimes_{\Z}{k}$. Let $\pi\colon \mathscr{X} \to M_{2d, k}$ be the universal family. \par
  Now we recall the results of van der Geer and Katsura on the description of the locus 
 $${M^{(h)}_{2d, k}}=\{\, s \in{M_{2d, k}} \mid {h(\mathscr{X}_{s})\geq{h}}\, \}$$
  for $1\leq{h}\leq11$. The locus ${M^{(h)}_{2d, k}}$ is a closed substack of $M_{2d, k}$. The locus $M^{(11)}_{2d, k}$ is called the \textit{supersingular locus}. The loci ${M^{(h)}_{2d, k}}$ $(1\leq{h}\leq11)$ form the so-called \textit{height stratification} of the moduli stack $M_{2d, k}$.  
  
  Let $\lambda:=\pi_*({\Omega^2_{\mathscr{X}/{M_{2d, k}}}})$ be the \textit{Hodge bundle} on $M_{2d, k}$. For each $1 \leq h \leq 10$, van der Geer and Katsura constructed an $\O_{{M^{(h)}_{2d, k}}}$-linear morphism of line bundles
 $$\phi_h\colon ({R^2\pi_*{{\O}_{\mathscr{X}}}})^{(p^h)}\vert_{M^{(h)}_{2d, k}} \to {R^2\pi_*{{\O}_{\mathscr{X}}}}\vert_{M^{(h)}_{2d, k}}$$
   on ${M^{(h)}_{2d, k}}$ such that the locus where $\phi_h$ vanishes coincides with ${M^{(h+1)}_{2d, k}}$ \cite[Theorem 15.1]{van-katsura}. Here, $({R^2\pi_*{\O_{\mathscr{X}}}})^{(p^h)}$  is the pullback of ${R^2\pi_*{\O_{\mathscr{X}}}}$ by the $h$-th power of the absolute Frobenius morphism $F\colon M_{2d, k} \to M_{2d, k}$. Therefore, for each $1\leq{h}\leq10$, there is a global section $$g_h\in\Gamma({M^{(h)}_{2d, k}},\lambda^{\otimes{(p^h-1)}})$$ on $M^{(h)}_{2d, k}$ whose zero locus coincides with ${M^{(h+1)}_{2d, k}}$. \par
  When $p \geq5$, Maulik showed the following important theorem in \cite{Maulik}. Madapusi Pera proved it when $p \geq3$ in \cite{Madapusi}.
 \begin{thm}[{\cite[Corollary 5.16]{Madapusi}}, {\cite[Theorem 5.1]{Maulik}}]\label{positivity} Assume $p\geq3$. Then the Hodge bundle $\lambda$ is ample on the polarized locus $M^{\circ}_{2d, k}$. Furthermore, for any morphism $g\colon C \to M_{2d, k}$ from a proper smooth curve $C$ over $k$ such that
 \begin{itemize}
 \item $C$ is not contracted in the map to the coarse moduli space, and
 \item $C$ meets the polarized locus $M^{\circ}_{2d, k}$ nontrivially,
 \end{itemize}
 the pullback of the Hodge bundle $\lambda$ on $M_{2d, k}$ to $C$ has positive degree. 
   \end{thm}
   The following corollary is stated in \cite[Theorem 15.3]{van-katsura}, \cite[Corollary 5.5]{Maulik} in the polarized case, but we can prove it in the quasi-polarized (and generically polarized) case in a similar way.
   \begin{cor}\label{jump}
 Assume $p\geq3$. Let $C$ be a proper smooth curve over $k$, and $f\colon \mathscr{X}\rightarrow{C}$ a non-isotrivial family of $K3$ surfaces with a primitive quasi-polarization $\mathscr{\xi}$ of degree $2d$ such that the generic fiber $(\mathscr{X}_{\eta}, \xi({\eta}))$ is a polarized $K3$ surface. Then, either
 \begin{itemize}
\item the heights of the geometric fibers $\mathscr{X}_t$ are not constant, or
\item all geometric fibers $\mathscr{X}_t$ are supersingular.
\end{itemize}
\end{cor}
\begin{proof}
If the heights of the geometric fibers $\mathscr{X}_{t}$ are constant, and they are equal to $h$ ($1\leq{h}\leq10$), then the image of $C$ is contained in ${M^{(h)}_{2d, k}}\setminus{M^{(h+1)}_{2d, k}}$. The pullback of $\lambda^{\otimes{(p^h-1)}}$  to $C$ is trivial since the pullback of $g_h$ is a global section which is everywhere nonzero on $C$. However the pullback of $\lambda$ to $C$ has positive degree by Theorem \ref{positivity}, which is absurd.
\end{proof}
\begin{rem}
\rm{The polarized locus $M^{\circ}_{2d, k}$ is separated \cite[Theorem 4.3.3]{Rizov06}, but $M_{2d, k}$ is not separated. The non-separatedness of $M_{2d, k}$ is related to the existence of flops; see \cite[last paragraph in p.\,2362]{Maulik}. So we do not expect ${M^{(h)}_{2d, k}}\setminus{M^{(h+1)}_{2d, k}}$ is ``quasi-affine" in a reasonable sense. However, it does not cause any problem to obtain Corollary \ref{jump} thanks to Theorem \ref{positivity}. Compare the proof of Corollary \ref{jump} with the proof of \cite[Theorem 15.3]{van-katsura}.}
\end{rem}

\section{Proof of Theorem \ref{main}}\label{proof}  
In this section, for a discretely valued field $L$, we denote its valuation ring by $\O_L$, the generic point by $\eta \in \Spec{\O_L}$, and the closed point by $s \in \Spec{\O_L}$. \par
Let $k$ be an algebraically closed field of characteristic $p\geq5$, and $C$ a proper smooth curve over $k$ with function field $K=k(C)$. Let $(X,\mathscr{L})$ be a non-isotrivial primitively polarized $K3$ surface over $K$ of degree $2d$ with $p>18d+4$ and $3\leq{h(X)}\leq10$. Then there exist a non-empty Zariski open set $U\subset{C}$ and a family of $K3$ surfaces $$\mathscr{X}_{U} \to {U}$$ with a polarization $\xi_{U}$ of degree $2d$ which extends $(X,\mathscr{L})$. The set of closed points outside $U$ is denoted by $C\setminus{U}=\lbrace{v_1},\dotsc,{v_m}\rbrace$. Take a closed point $v_i\in{C\setminus{U}}$. We put $K_i:=K$ and consider it as a discretely valued field with respect to the valuation defined by ${v_i}$. \par
First, we shall extend $\mathscr{X}_{U}$ to a quasi-polarized family of $K3$ surfaces over $C$, after replacing $C$ by a finite covering of it. This step is now well-understood thanks to Maulik's work \cite{Maulik}. We briefly recall his argument with a minor modification; see Remark \ref{rem}. In \cite[Section 4.3]{Maulik}, Maulik constructed a projective scheme $${\mathscr{X}'_i} \to {\Spec{\O_{K'_i}}}, $$after taking the base change to a finite extension ${K'_i}/{K_i}$, which satisfies the following conditions: 
\begin{enumerate}
\item $\mathscr{X}'_i$ is Cohen-Macaulay, $\Q$-factorial, and regular away from finitely many points,
\item the generic fiber $\mathscr{X}'_{i,\eta}$ is isomorphic to $X_{K'_i}:=X\otimes_{K_i}{K'_i}$,
\item the closed fiber $\mathscr{X}'_{i,s}$ is reduced and irreducible components of $\mathscr{X}'_{i,s}$ are normal,
\item the singularities of $\mathscr{X}'_{i,s}$ are rational double points or of normal crossing type, and
\item the relative canonical bundle $K_{{\mathscr{X}'_i}/{\O_{K'_i}}}$ is trivial and there is a nef $\Q$-divisor $H$ on $\mathscr{X}'_i$ such that
$${H}\vert_{\mathscr{X}'_{i,\eta}}\thicksim_{\Q}\gamma\mathscr{L}_{K'_i}$$
for some positive rational number $\gamma\in\Q_{>0}$, where $\mathscr{L}_{K'_i}:=\mathscr{L}\otimes_{K_i}{K'_i}$.

\end{enumerate}
 Here we use the fact that $\mathscr{L}^{\otimes{3}}$ is very ample \cite{Saint} and the assumption $p>18d+4$ to apply Saito's results \cite{Saito} on the construction of projective semistable models.

\begin{claim}
If $h(X)\geq3$, then there exist a finite extension of discretely valued fields  ${L_i}/{K_i}$ and a primitively quasi-polarized family of $K3$ surfaces 
$$\mathscr{X}_i \to \Spec{\mathscr{O}_{L_i}},$$
which extends $(X_{L_i},\mathscr{L}_{L_i}):=(X\otimes_{K_i}{L_i}, \mathscr{L}\otimes_{K_i}{L_i})$.
\end{claim}
\begin{proof}
We shall apply Artin's results on the simultaneous resolution of rational double points in the category of algebraic spaces \cite[Theorem 2]{Artin}. Then we find a finite extension of discretely valued fields ${L_i}/{K'_i}$, a proper smooth family of $K3$ surfaces
$$\mathscr{X}_i \to \Spec{\mathscr{O}_{L_i}},$$
and a simultaneous resolution over $\Spec{\mathscr{O}_{L_i}}$
$$\mathscr{X}_i \to {\mathscr{X}'_i}\otimes_{\mathscr{O}_{K'_i}}{\mathscr{O}_{L_i}}.$$
 Here $\mathscr{X}_i$ is an algebraic space over $\Spec{\mathscr{O}_{L_i}}$, and $\mathscr{X}_i \to {\mathscr{X}'_i}\otimes_{\O_{K'_i}}{\mathscr{O}_{L_i}}$ induces an isomorphism on the generic fiber and the minimal resolution of rational double points on the closed fiber. \par 
We apply Nakkajima's results \cite[Proposition 3.4]{Nakkajima}, and see that the closed fiber $\mathscr{X}_{i,s}$ is a combinatorial $K3$ surface, which means one of the following types:
\begin{enumerate}
\item smooth $K3$ surface (type I),
\item two rational surfaces joined by a chain of ruled surfaces over an elliptic curve (type II),
\item union of rational surfaces, whose dual graph of the configuration is a triangulation of $S^2$ (type III).
\end{enumerate}\par
In \cite[Proposition 3]{Rudakov}, Rudakov, Zink, and Shafarevich showed that the height $h(\mathscr{X}_{i,s})$ of the closed fiber satisfies $h(\mathscr{X}_{i,s})\geq{h(\mathscr{X}_{i,{\eta}})}$. Moreover, $1\leq{h(\mathscr{X}_{i,s})}\leq2$ is satisfied if the closed fiber $\mathscr{X}_{i,s}$ is not smooth. So by our assumption $h(X)\geq3$, the family of $K3$ surfaces $\mathscr{X}_i \to {\Spec{\mathscr{O}_{L_i}}}$ is necessarily smooth. \par
We shall show that the line bundle $\mathscr{L}_i$ on $\mathscr{X}_i$ which extends $\mathscr{L}_{L_i}$ is a quasi-polarization. We follow Maulik's argument as in \cite[Lemma 4.10]{Maulik}. We denote the natural morphism $\mathscr{X}_i \to \mathscr{X}'_i$ by $g$. The pullback $g^*{H}$ of the nef $\Q$-divisor $H$ is also nef. Since we have ${H}\vert_{\mathscr{X}'_{i,\eta}}\thicksim_{\Q}\gamma\mathscr{L}_{K'_i}$ on $\mathscr{X}'_{i,\eta}$, we also have
$$g^*{H}\thicksim_{\Q}\gamma\mathscr{L}_i$$
and hence $\mathscr{L}_i$ is nef by \cite[Lemma 5.12]{Maulik}. The bigness of $\mathscr{L}_i$ follows from the fact that the Euler-Poincar$\acute{\rm{e}}$ characteristics are locally constant in a flat family.
Since the order of a torsion element in the cokernel of $$\Pic(\mathscr{X}_{i, \overline{\eta}}) \hookrightarrow \Pic(\mathscr{X}_{i,s})$$ is a power of $p$ \cite[Proposition 3.6]{Maulik-Poonen} and $$p>18d+4>2d=(\mathscr{L})^2=(\mathscr{L}_i)^2,$$ the line bundle $\mathscr{L}_i$ is also primitive on the closed fiber.
\end{proof}
\vspace*{0.1in}
{\sc Proof of Theorem \ref{main}}. We now give a proof of Theorem \ref{main}. The family $\mathscr{X}_{U} \to {U}$ can be extended to a primitively quasi-polarized family of $K3$ surfaces $$\mathscr{X} \to {C}$$ over a proper smooth curve $C$, after replacing it by a finite covering of it \cite[Lemma 7.2]{Maulik}. \par
By Corollary \ref{jump}, there exists a closed point $v \in{C}$ such that the height of $X$ jumps at $v$. So we achieve Theorem \ref{main}.  \qed\par
\vspace*{0.1in}
{\sc Proof of Corollary \ref{supersingular-reduction}}. Assume $$\rho(X_{\overline{K}})\geq21-2h(X),$$
 and take a closed point $v\in{C}$ such that the height of $X$ jumps at $v$, namely, $h(\mathscr{X}_v)>h(X)$. Then $X$ necessarily has supersingular reduction at $v$. Indeed, if $h(\mathscr{X}_v)\neq\infty$, then we have 
$$22-2h(\mathscr{X}_v)\geq\rho(\mathscr{X}_v)\geq\rho(X_{\overline{K}})\geq21-2h(X).$$
Hence we have $h(X)\geq{h(\mathscr{X}_v)}-1/2$, which is absurd. \qed
\vspace*{0.1in}

In Theorem \ref{main}, we assume $p>18d+4$ and $h(X)\geq3$, but one expects that the following conjecture is true without any assumptions on $p$, $d$, and $h(X)$.
 \begin{conj}[{\cite[p.288]{van-katsura}, \cite[p.2390]{Maulik}}]\label{conj}
  Let $k$ be an algebraically closed field of positive characteristic $p>0$, and $C$ a proper smooth curve over $k$ with function field $K=k(C)$. Let $X$ be a non-isotrivial non-supersingular $K3$ surface over $K$. Then, after replacing $C$ by a finite covering of it, there exists a semistable family of combinatorial $K3$ surfaces $\mathscr{X} \to C$ which extends $X$ and satisfies $h(\mathscr{X}_v)>h(X)$ at some closed point $v\in{C}$.
 \end{conj}
A natural strategy to prove Conjecture \ref{conj} is to generalize Maulik's results (Theorem \ref{positivity}) to an appropriate compactification of the moduli stack $M_{2d, k}$, but such results are not currently available.
\begin{rem}\label{rem}
\rm{In the proof of Theorem \ref{main}, the construction of a quasi-polarized family of $K3$ surfaces $\mathscr{X} \to {C}$ is basically the same as in Maulik's proof of \cite[Theorem 4.2]{Maulik}. But there is a slight technical difference. In \cite[Theorem 4.2]{Maulik}, Maulik considered families of \textit{supersingular} $K3$ surfaces, and he constructed families $\mathscr{X} \to {C}$ \textit{as schemes}. In fact, he applied Artin's simultaneous resolution over the completion of the local ring at each $v_i\in{C}\setminus{U}$ in order to show that there is no rational double point in the closed fiber $\mathscr{X}'_{i, v_i}$; see the proof of  \cite[Lemma 4.10]{Maulik}. (This argument works only for supersingular $K3$ surfaces.) On the other hand, we apply Artin's simultaneous resolution \cite{Artin} over (the henselization of) the local ring at each $v_i\in{C}\setminus{U}$ rather than the completion of it. (This is possible since Artin's results in \cite{Artin} are valid over henselian (not necessarily complete) discrete valuation rings.) Consequently, in the proof of Theorem \ref{main}, we construct families $\mathscr{X} \to {C}$ \textit{as algebraic spaces}. Note that the example discussed by Matsumoto in \cite[Example 5.2]{Matsumoto} shows that one does not expect to construct families $\mathscr{X}\to{C}$ as schemes in general. For details, compare the proof of Theorem \ref{main} with Maulik's proof of \cite[Theorem 4.2, Lemma 4.7]{Maulik}.}
\end{rem}
\section{Deformations of $K3$ surfaces with constant height and Picard number}\label{deformation}
 In the rest of this paper, we shall give a proof of Theorem \ref{main2}. We fix an algebraically closed field $k$ of characteristic $p>0$. In this section we show the following.
\begin{prop}\label{defo}
Let $(X_0,\L_0)$ be a polarized $K3$ surface over $k$ satisfying $2\leq{h(X_0)}\leq10$. Then there exists a non-isotrivial polarized $K3$ surface $(X,\mathscr{L})$ over the function field of a proper smooth curve over $k$ such that $\rho(X)=\rho(X_0)$, $h(X)=h(X_0)$, and $(\mathscr{L})^2=(\mathscr{L}_0)^2.$ 
\end{prop}
\begin{proof}
We use the deformation theory to construct a non-isotrivial family. Similar arguments may be found in \cite[Section 4]{Lieblich-Maulik}. Let $\pi\colon \mathscr{X} \to \mathscr{S}$ be the versal formal deformation of $X_0$ over $k$. It is known that $\mathscr{S}$ is formally smooth of dimension $20$ over $k$; so we have 
$$\mathscr{S}\simeq{\Spf{k[[t_1,\dotsc, t_{20}]]}}.$$
See \cite[Corollary 1.2]{Deligne}.\par
 Let $\L_1,\dotsc, \L_{\rho(X_0)}$ be a $\Z$-basis of $\Pic(X_0)$. Then the deformation space of the pair $(X_0,\L_i)$ is defined by an equation $${f_i}\in{{k[[t_1,\dotsc, t_{20}]]}}.$$
 See \cite[Proposition 1.5]{Deligne}. A linear combination of $\L_1,\dotsc, \L_{\rho(X_0)}$ is ample since $X_0$ is projective. Hence by Grothendieck's algebraization theorem, we get a universal family 
 $$\mathscr{X}' \to S'=\Spec({{k[[t_1,\dotsc, t_{20}]]}/({f_1,\dotsc,f_{\rho(X_0)}}}))$$
 of deformations of the tuple $(X_0, \L_1,\dotsc, \L_{\rho(X_0)}).$ For each $1\leq{h}\leq10$, let
$$S'^{(h)}:=\{\, s\in{S'} \mid {h(\mathscr{X}'_{s})\geq{h}}\,\}.$$
In \cite{van-katsura}, van der Geer and Katsura showed that there are elements 
$$g_1,\dotsc, g_{h(X_0)-1}\in{k[[t_1,\dotsc, t_{20}]]},$$ 
such that, for each $2\leq{h}\leq{h(X_0)}$, the closed subscheme $S'^{(h)}\subset{S'}$ is written as 
$$S'^{(h)}=\Spec({{k[[t_1,\dotsc, t_{20}]]}/({f_1,\dotsc, f_{\rho(X_0)},g_1,\dotsc,g_{h-1}}})).$$\par
Then the $K3$ surface $\mathscr{X}'_s$ corresponding to a geometric point $s$ of $S'^{(h(X_0))}$ satisfies 
$$\rho(\mathscr{X}'_s)=\rho(X_0),\qquad h(\mathscr{X}'_s)=h(X_0).$$
Moreover, since $\mathscr{L}_0$ is a $\Z$-linear combination of $\L_1,\dotsc, \L_{\rho(X_0)}$, the $K3$ surface $\mathscr{X}'_s$ has a line bundle $\mathscr{L}_s$ which lifts $\mathscr{L}_0$. The line bundle $\mathscr{L}_s$ is ample with $(\mathscr{L}_s)^2=(\mathscr{L}_0)^2$. 
Since 
$$\rho(X_0)+h(X_0)-1\leq21-h(X_0)<20,$$ 
the dimension of $S'^{(h(X_0))}$ is positive. Hence we find a non-isotrivial family over a complete local ring of dimension $\geq1$. After passing limits and cutting by hyperplanes, we find a non-isotrivial polarized $K3$ surface $(X, \mathscr{L})$ over a function field such that $\rho(X)=\rho(X_0)$, $h(X)=h(X_0)$, and  $(\mathscr{L})^2=(\mathscr{L}_0)^2$.
\end{proof}

By Proposition \ref{defo}, in order to construct non-isotrivial $K3$ surfaces $X$ over function fields satisfying $\rho(X)=22-2h(X)$ and $2\leq{h(X)}\leq10$, it suffices to construct such $K3$ surfaces over $k$. 
\par
We note that non-isotrivial $K3$ surfaces $X$ with $\rho(X)=22-2h(X)$ do \textit{not} exist when $p$ is odd and $h(X)=1$. This follows from the following result recently proved by Jang and Liedtke, independently.

\begin{thm}[{\cite[Theorem 3.7]{Jang}, \cite[Theorem 2.6]{Liedtke}}]\label{20}
Let $X$ be a $K3$ surface with $\rho(X)=20$ over an algebraically closed field $k$ of characteristic $p\geq3$. Then X is defined over a finite field.
\end{thm}
\begin{proof}
See \cite[Theorem 3.7]{Jang} and \cite[Theorem 2.6]{Liedtke}.
\end{proof}

\section{The existence of non-isotrivial $K3$ surfaces with large Picard number}\label{existence}

Assuming the existence of potential semistable reduction, Taelman conditionally proved the existence of $K3$ surfaces over finite fields with given $L$-function, up to finite extensions of the base field \cite{Taelman}. When the characteristic of the base field is not too small, the author recently proved that Taelman's results hold unconditionally \cite{Ito}. As an application, the author proved the following theorem.

\begin{thm}[{\cite[Theorem 7.4]{Ito}}]\label{existenceprop}
Let $p$ be a prime number with $p\geq5$. Let $\rho \in (2\Z)_{>0}$ be a positive even integer and $h\in\Z_{>0}$ a positive integer with $$\rho \leq 22-2h.$$
Then there exists a $K3$ surface $X$ over $\overline{\F}_p$ such that $\rho(X)=\rho$ and $h(X)=h$.
\end{thm}
\begin{proof}
See \cite[Theorem 7.4]{Ito}.
\end{proof}
\begin{prop}\label{2,5}
Let $p$ be a prime number and $h$ a positive integer with $1 \leq h \leq 10$. There exists a $K3$ surface $X$ over $\overline{\F}_p$ with $\rho(X)=22-2h(X)$ and $h(X)=h$ if at least one of the following conditions holds:
\begin{itemize}
\item $p=3$ and $h=10,$ or
\item $p\geq5$.
\end{itemize}
\end{prop}
\begin{proof}
 When $p\geq5$ and $1\leq{h}\leq9,$ we have $22-2h \geq 4$. Then, there exists a $K3$ surface $X$ over $\overline{\F}_p$ with $\rho(X)=22-2h(X)$ and $h(X)=h$ by Theorem \ref{existenceprop}.\par
 When $p\geq3$ and $h=10,$ it is well-known that there exists a $K3$ surface $X$ over $\overline{\F}_p$ such that $h(X)=h$ by \cite[Corollary 2.2]{Shioda} and \cite[Theorem 14.2]{van-katsura}. Then, we have $\rho(X) \leq 22-2h=2$. By the Tate conjecture for $K3$ surfaces over finite fields, the Picard number $\rho(X)$ is even; see \cite[the paragraph following Definition 0.3]{Artin74}. (See also \cite[Chapter 17, Corollary 2.9]{Huybrechts}.) Hence we have $\rho(X)=2$.
\end{proof}
Combining Proposition \ref{2,5} with Proposition \ref{defo}, we obtain Theorem \ref{main2}.
\vspace*{0.1in}

{\sc Proof of Theorem \ref{main2}}.
By Proposition \ref{2,5}, there exists a $K3$ surface $X_0$ over $\overline{\F}_p$ with $\rho(X_0)=22-2h(X_0)$ and $h(X_0)=h$. By Proposition \ref{defo}, we have a non-isotrivial $K3$ surface $X$ over the function field of a proper smooth curve over $\overline{\F}_p$ with $\rho(X)=\rho(X_0)$ and $h(X)=h(X_0)$. In particular, $X$ satisfies $\rho(X)=22-2h(X)$ and $h(X)=h$.\hfill \qed

\vspace*{0.1in}

\begin{exm}\label{Yui-Goto}
\rm{In \cite{Yui}, \cite{Goto}, Yui and Goto calculated the heights of the formal Brauer groups of $K3$ surfaces over finite fields concretely. In their list, we can find concrete examples of $K3$ surfaces $X$ over $\overline{\F}_p$ satisfying $\rho(X)=22-2h(X)$. }
\end{exm}
\begin{rem}\label{maincond}
\rm{Using Taelman's theorem \cite[Theorem 3]{Taelman} and Proposition \ref{defo}, it is easy to see that Theorem \ref{main2} holds for any $p$ and any $2\leq{h}\leq{10}$ if we assume \cite[Assumption ($\star$)]{Liedtke-Matsumoto}.}
\end{rem}

\subsection*{Acknowledgements}
The author is deeply grateful to my adviser, Tetsushi Ito, for his kindness, support, and advice. He gave me a lot of invaluable suggestions and pointed out some mistakes in an earlier version of this paper. The author also would like to thank Yuya Matsumoto for helpful suggestions and comments. Moreover the author would like to express his gratitude to the anonymous referees for sincere remarks and comments.

\end{document}